\documentclass[12pt, reqno]{amsart}
\usepackage{amsmath, amsthm, amscd, amsfonts, amssymb, graphicx, color,cite}
\usepackage[bookmarksnumbered, colorlinks, plainpages]{hyperref}
\hypersetup{colorlinks=true,linkcolor=red, anchorcolor=green, citecolor=cyan, urlcolor=red, filecolor=magenta, pdftoolbar=true}

\newtheorem{thm}{Theorem}[section]
\newtheorem{cor}[thm]{Corollary}
\newtheorem{lem}[thm]{Lemma}
\newtheorem{prop}[thm]{Proposition}
\theoremstyle{definition}

\theoremstyle{remark}

\numberwithin{equation}{section}

\begin{document}
\setcounter{page}{1}

\title[Successive radii and Orlicz Minkowski addition]{Successive radii and Orlicz Minkowski addition}

\author[F. Chen, C. Yang,  M. Luo]{Fangwei Chen$^1$, Congli Yang$^2$,  Miao Luo$^{2}$,}

\address{1. Department of Mathematics and Statistics, Guizhou University of Finance and Economics,
Guiyang, Guizhou 550004, People's Republic of China}

\email{cfw-yy@126.com; chen.fangwei@yahoo.com}
\address{2, 3. School of Mathematics and Computer Science, Guizhou Normal
University, Guiyang, Guizhou 550001, People's Republic of China.}
\email{yangcongli@gznu.edu.cn}
\thanks{The  work is supported in part by CNSF (Grant No. 11161007, Grant No. 11101099), Guizhou (Unite) Foundation for Science and Technology (Grant
No. [2014] 2044, No. [2012] 2273, No. [2011] 16), Guizhou Technology Foundation for Selected Overseas Chinese Scholar and Doctor foundation of Guizhou Normal University.}

%\dedicatory{This paper is dedicated to Professor ABCD}



%\date{Received: xxxxxx; Revised: yyyyyy; Accepted: zzzzzz.
%\newline \indent $^{*}$ Corresponding author}

\begin{abstract}
In this paper, we deal with the successive inner and outer radii with respect to Orlicz Minkowski sum. The upper and lower bounds for the radii of the Orlicz Minkowski sum of two convex bodies are established.
\end{abstract} \maketitle

Let $\mathcal K^n$ denote the set of convex bodies (compact, convex subsets with nonempty interiors) in Euclidean $n$-space, $\mathbb R^n$. Let $\langle\cdot,\cdot\rangle$ and $\rVert \cdot\lVert_2$ denote the standard inner product and the Euclidean norm in $\mathbb R^n$, respectively. Denote by $e_i$ ($i=1,\cdots,n$), the orthogonal unit vectors in $\mathbb R^n$. The $n$-dimensional unit ball and its boundary, i.e., the $(n-1)$-dimensional unit sphere is denoted by $B_n$ and $S^{n-1}$, respectively.

The volume of a set $K\in\mathcal K^n$, i.e., its $n$-dimensional Lebesgue measure is denoted by $|K|$. The set of all $i$-dimensional linear subspaces of $\mathbb R^n$ is denoted by $\mathcal L^n_i$. For $L\in\mathcal L^n_i$, $L^\bot$ denotes its orthogonal complement. Let $K\in \mathcal K^n$, $L\in\mathcal L^n_i$, the projection of $K$ onto $L$ is denoted by $K|L$.

Following the traditional notations, we use $D(K)$, $\omega (K)$, $R(K)$ and $r(K)$ to denote the diameter, minimal width, circumradius and inradius of a convex body $K$, respectively. The behavior of the diameter, minimal width, circumradius and inradius with respect to the Minkowski sum is well known(see \cite{sch-con1993}), namely
\begin{align*}
  D(K+K')\leq D(K)+D(K'),\,\,\,\,\,\,\,\,\,\,\,\, \omega(K+K')\leq \omega(K)+\omega(K'),\\
  R(K+K')\leq R(K)+R(K'),\,\,\,\,\,\,\,\,\,\,\,\, r(K+K')\leq r(K)+r(K').
\end{align*}

Let $K\in\mathcal K^n$, and $i=1,2,\cdots,n$, the successive outer and inner radii of $K$ are defined as
\begin{align*}
  R_i(K)=\min_{L\in\mathcal L^n_i}R(K|L),\,\,\,\,\,\,\,\,\,\,\,r_i(K)=\max_{L\in\mathcal L^n_i}\max_{x\in L^\bot}r(K\cap(x+L): x+L).
\end{align*}
Notice that $R_i(K)$ is the smallest radius of a  solid cylinder with $i$-dimensional spherical cross section containing $K$, and $r_i(K)$ is the radius of the greatest $i$-dimensional ball contained in $K$.
It is clear that the outer radii are increasing in $i$, whereas the inner radii are decreasing in $i$, and the following hold (see \cite{bon-fen-the1934}).
\begin{align*}
  R_n(K)=R(K),\,\,\,\,\,\,R_1(K)=\frac{\omega(K)}{2},\,\,\,\,\,\,\,r_n(K)=r(K),\,\,\,\,\,\,\,r_1(K)=\frac{D(K)}{2}.
\end{align*}
The first systematic study of the successive radii was developed in \cite{bet-hen-est1992}, and one can refer \cite{bet-hen-a1993,bra-rad2005,gon-her-suc2012,gon-her-on2014,gri-kle-inn1992,hen-a1992,hen-her-suc2009,hen-her-int2008} and references within for more details.

The radii of convex bodies which connected the Minkowski sum (or $L_p$-Minkowski sum) are studied by Gonz\'alez and Hern\'andez Cifre \cite{gon-her-suc2012,gon-her-on2014}.

Beginning with the articles \cite{hab-lut-yan-zha-the2010,lut-yan-zha-orl-cen2010,lut-yan-zha-orl-pro2010} of Haberl, Lutwak, Yang ang Zhang, a more wide extension of the $L_p$ Brunn-Minkowski theory emerged.  Recently, in a paper of Gardner, Hug and Weil \cite{gar-hug-wei-the2014}, a systematic studies are made on the Orlicz Minkowski addition, the Orlicz Brunn-Minkowski inequality and Orlicz Minkowski inequality are obtained. The Orlicz Brunn-Minkowski theory are established. See, e.g., \cite {bor-str2013,bor-lut-yan-zha-the2012,bor-lut-yan-zha-the2013,che-zho-yan-on2011,che-yan-zho-the2014,gar-hug-wei-the2014,
hab-lut-yan-zha-the2010,hua-he-on2012,zhu-zho-xu-dua2014,zou-xio-orl2014} about the Orlicz Brunn-Minkowski theory. In this context, the main goal of this paper is to seek the relations of the radii for Orlicz Minkowski sum.

Throughout this paper, let $ \mathcal C$ be the class of convex, strictly increasing functions $\varphi:[0,\infty)\rightarrow [0,\infty)$ satisfying $\varphi(0)=0$ and $\varphi(1)=1$. Here the normalization is a matter of convenience and other choices are possible.

Let $\mathcal K^n_o$ denote the class of convex bodies containing the origin. $K,L\in\mathcal K^n_o$, $\varphi\in \mathcal C$, the Orlicz Minkowski sum of $K$ and $L$ is the convex body $K+_\varphi L$ with support function
\begin{align*}
  h_{ K+_\varphi L}(x)=\inf\left\{\lambda>0:\varphi\left(\frac{h_K(x)}{\lambda}\right)+\varphi\left(\frac{h_L(x)}{\lambda}\right)\leq 1\right\}.
\end{align*}

If $\varphi(t)=t^p$, $p\geq 1$, then $K+_\varphi L$ is precisely the $L_p$ Minkowski sum $K+_p L$.

For the successive outer radii of Orlicz Minkowski sum, we establish the following theorem.
\begin{thm}\label{out-rad-ine-thm}
  Let $\varphi\in \mathcal C$ and $K, K'\in \mathcal K^n_o$. Then
  \begin{align}
   \label{out-rad-ine-1} &2\varphi^{-1}(1/2)R_1(K+_\varphi K')\geq R_1(K)+R_1(K'),\\
   \label{out-rad-ine-2} & 2\sqrt2\varphi^{-1}(1/2)R_i(K+_\varphi K')\geq R_i(K)+R_i(K'), \,\,\,\,\,\,\,\,\,\,\,i=2,\cdots, n.
  \end{align}
  All inequalities are best possible.
  \end{thm}
We also prove that there is non-existence the reverse inequalities for the successive outer radius excepted $R_n$. That is
\begin{prop}\label{rev-out-rad-pro}
Let $\varphi\in\mathcal C$ and $K, K'\in \mathcal K^n_o$, for $i=1,\cdots n-1$, there exists no constant $c>0$ such that
\begin{align*}
  cR_i(K+_\varphi K')\leq R_i(K)+R_i(K').
\end{align*}
\end{prop}

Similarly, for the successive inner radii of Orlicz Minkowski sum, we obtain
\begin{thm}\label{in-rad-ine-thm}
  Let $\varphi\in \mathcal C$ and $K, K'\in \mathcal K^n_o$. Then
  \begin{align}
    \label{inn-rad-ine-1}&2\varphi^{-1}(1/2)r_1(K+_\varphi K')\geq r_1(K)+r_1(K'),\\
    \label{inn-rad-ine-2}& 2\sqrt2\varphi^{-1}(1/2)r_i(K+_\varphi K')\geq r_i(K)+r_i(K'), \,\,\,\,\,\,\,\,\,\,\,i=1,\cdots, n-1.
  \end{align}
  All inequalities are best possible.
\end{thm}
The analogous Proposition \ref{rev-out-rad-pro} for the successive inner radii are
\begin{prop}
Let $\varphi\in\mathcal C$ and $K, K'\in \mathcal K^n_o$, for $i=2,\cdots,n$, there exists no constant $c>0$ such that
\begin{align*}
  cr_i(K+_\varphi K')\leq r_i(K)+r_i(K').
\end{align*}
\end{prop}
If we take $\varphi(t)=t^p$, $p\geq 1,$ these results are proved by Gonz\'alez and Hern\'{a}ndez Cifre
(see \cite{gon-her-on2014}). Specially, if $p=1$, it was shown in \cite{gon-her-suc2012}.

The second part of our result  is regard the Orlicz difference body. We obtain the following:
\begin{thm}
  Let $\varphi\in \mathcal C$, and $K\in \mathcal K^n_o$, for all $i=1,2\cdots, n$, then
\begin{align}\
\label{out-rad-ine-dif-1} \frac{\sqrt 2}{2\varphi^{-1}(1/2)}\sqrt{\frac{i+1}{i}}R_i(K)\leq R_i(K+_\varphi (-K)) \leq 2R_i(K),\\
\label{inn-rad-ine-dif-1}\frac{1}{\varphi^{-1}(1/2)}r_i(K)\leq r_i(K+_\varphi (-K))<2(i+1)r_i(K).
\end{align}
\end{thm}
If we take $\varphi(t)=t^p$, $p\geq 1$, these results are proved by Gonz\'alez and Hern\'{a}ndez Cifre
(see \cite{gon-her-suc2012, gon-her-on2014}).

The paper is organized as follows. In section 1, we introduce the Orlicz Minkowski sum and show some of their properties. The proof of the results of successive outer and inner radii for Orlicz Minkowski sum are given in Section 2. Section 3 deals with the successive radii for Orlicz difference body.

\section{Orlicz Minkowski addition}
In this section, some basic definitions and notations about Orlicz Minkowski sum and some of their properties are introduced.

Let $\varphi\in \mathcal C$, $x=(x_1,x_2,\cdots,x_n)\in \mathbb R^n$, the Orlicz norm $\lVert x\rVert _\varphi$ of a point $x\in \mathbb R^n$ is defined as
\begin{align}\label{orl-nor-for}
\lVert x\rVert_\varphi=\inf\left\{\lambda>0: (x_1,x_2,\cdots,x_n)\in \mathbb R^n, \sum^{n}_{i=1}\varphi\left(\frac{|x_i|}{\lambda}\right)\leq 1\right\}.
\end{align}
Note that, if take $\varphi(t)=t^p$, then $\rVert \cdot\lVert_\varphi$ is precisely the $L_p$ norm $\rVert \cdot\lVert_p$, if $p=2$, it is the Euclidean norm $\rVert \cdot\lVert_2$.

Let  $x\in \mathbb R^n$, then
\begin{align*}
  \lVert x\rVert_\varphi\geq 0,
\end{align*}
and for $c>0$,  we have
\begin{align*}
  \lVert cx\rVert_\varphi=c \lVert x\rVert_\varphi.
\end{align*}

The Orlicz ball is defined as
\begin{align*}
  B^\varphi_n=\left\{x=(x_1,\cdots, x_n)\in\mathbb R^n: \lVert x\rVert_\varphi\leq 1\right\}.
\end{align*}
We have the following Lemma.
\begin{lem}\label{orl-nom-ineq}
Let $\varphi_1, \varphi_2\in \mathcal C$, $x\in \mathbb R^n$. If $\varphi_1\leq\varphi_2$, then
\begin{align*}
\lVert x\rVert_{\varphi_1}\leq \lVert x\rVert_{\varphi_2},\,\,\,\,\,\,and\,\,\,\,\,\,\,\, B^{\varphi_2}_n\subseteq B^{\varphi_1}_n.
\end{align*}
\end{lem}
\begin{proof}
Let $x\in \mathbb R^n$, set $\lVert x\rVert_{\varphi_i}=\lambda_i$ $(i=1,\,2)$, by the definition (\ref{orl-nor-for}), we have $
\sum^{n}_{i=1}\varphi_i\left(\frac{|x_i|}{\lambda_i}\right)\leq 1.$
Since $\varphi$ is strictly increasing, then $\lambda\rightarrow \sum^{n}_{i=1}\varphi\left(\frac{|x_i|}{\lambda}\right)$ is strictly decreasing, so $\lVert x\rVert_\varphi=\lambda_i$ if and only if
$$\sum^{n}_{i=1}\varphi_i\left(\frac{|x_i|}{\lambda_i}\right)=1.$$
Note that $\varphi_1\leq \varphi_2$, then
$$\sum^{n}_{i=1}\varphi_2\left(\frac{|x_i|}{\lambda_1}\right)\geq1.$$
By the definition of $\lVert x \rVert_\varphi$ we have
$\lVert x \rVert_{\varphi_1}\leq \lVert x \rVert_{\varphi_2}$.

   Let $x\in \partial B^{\varphi_2}_n$, then
  $\lVert x\rVert_{\varphi_2}=1,$ by (\ref{orl-nom-ineq}), we have
  \begin{align*}
    \lVert x\rVert_{\varphi_1}\leq \lVert x\rVert_{\varphi_2}=1.
  \end{align*}
Then we have $x\in B^{\varphi_1}_n$, so we complete the proof.
    \end{proof}
Let $K,L\in\mathcal K^n_o$, $\varphi\in \mathcal C$, the Orlicz Minkowski sum of $K$ and $L$ is the convex body $K+_\varphi L$, with support function
\begin{align*}
  h_{ K+_\varphi L}(x)=\inf\left\{\lambda>0:\varphi\left(\frac{h_K(x)}{\lambda}\right)+\varphi\left(\frac{h_L(x)}{\lambda}\right)\leq 1\right\}.
\end{align*}
Since $\varphi$  is strictly increasing, then \begin{align*}\lambda\rightarrow\varphi\left(\frac{h_K(x)}{\lambda}\right)+\varphi\left(\frac{h_L(x)}{\lambda}\right),
 \end{align*}
  is strictly decreasing. So, equivalently,
  $h_{K+_\varphi  L}(u_0)=\lambda_0$  if and only if
 \begin{align}\label{equ-orl-con-add}
 \varphi\left(\frac{h_K(u_0)}{\lambda_0}\right)+\varphi\left(\frac{h_L(u_0)}{\lambda_0}\right)=1.
 \end{align}
For the body $K+_\varphi L$, we have the following results.
\begin{thm}\label{lin-rel-thm}
Let $\varphi$ and $\varphi_1, \varphi_2$  in $\mathcal C$, for $K, L\in \mathcal K^n_o$, then\\
(\romannumeral 1 ). If $\varphi_1\leq\varphi_2$  for all $x\in[0,1]$,  then $K+_{\varphi_1}L\subseteq K+_{\varphi_2}L, $\\
(\romannumeral 2). For $\varphi\in \mathcal C$, then $\frac{1}{2\varphi^{-1}(1/2)}K+L\subseteq K+_\varphi L\subseteq K+L,$\\
(\romannumeral 3). $conv(K\cup L)\subseteq K+_\varphi L,$\\
(\romannumeral 4). $K+_\varphi L \subseteq \frac{1}{\varphi^{-1}(1/2)}conv(K\cup L) .$
\end{thm}
\begin{proof}
  In order to prove (\romannumeral 1), we only need to show
  $$h_{K+_{\varphi_1}L}(u)\leq h_{K+_{\varphi_2}L}(u), \,\,\,\,\,\,\,\,\,\,\,\,\, for\,\,u\in S^{n-1}.$$
By formula (\ref{equ-orl-con-add}), we have
\begin{align*}
  \varphi_1\left(\frac{h_K(u)}{h_{K+_{\varphi_1}L}(u)}\right)+\varphi_1\left(\frac{h_L(u)}{h_{K+_{\varphi_1}L}(u)}\right)=1, \\  \varphi_2\left(\frac{h_K(u)}{h_{K+_{\varphi_2}L}(u)}\right)+\varphi_2\left(\frac{h_L(u)}{h_{K+_{\varphi_2}L}(u)}\right)=1.
\end{align*}
Since $\varphi_1\leq\varphi_2$, so
\begin{align*}
  \varphi_2\left(\frac{h_K(u)}{h_{K+_{\varphi_1}L}(u)}\right)+\varphi_2\left(\frac{h_L(u)}{h_{K+_{\varphi_1}L}(u)}\right)\geq1,
\end{align*}
which means $$h_{K+_{\varphi_1}L}(u)\leq h_{K+_{\varphi_2}L}(u).$$
So $K+_{\varphi_1}L\subseteq K+_{\varphi_2}L$ .

(\romannumeral 2). Let $Id$ denote the identity function on $[0,1]$, by the convexity of  $\varphi$  on $[0,1]$, for $x\in [0,1]$, we have
\begin{align*}
  \varphi(x)=\varphi(x\cdot 1+(1-x)0)\leq x\varphi(1)+(1-x)\varphi(0),
\end{align*}
so we have $\varphi(x)\leq x$, which means that $\varphi\leq Id$. On the other hand, when $\varphi=Id$, the Orlicz Minkowski sum $K+_\varphi L$ is precisely the Minkowski sum $K+L$.  Now by (\romannumeral 1),  we have
$ K+_{\varphi}L \subseteq K+L $.

For the left hand inclusion, let $u\in S^{n-1}$, by the definition of Orlicz Minkowski sum we have
  \begin{align*}
    \varphi\left(\frac{h_K(u)}{h_{K+_{\varphi}L}(u)}\right)+\varphi\left(\frac{h_L(u)}{h_{K+_{\varphi}L}(u)}\right)=1.
  \end{align*}
  The convexity of $\varphi$ implies
  \begin{align*}
 2\varphi\left(\frac{h_K(u)+h_L(u)}{2h_{K+_{\varphi}L}(u)}\right)\leq 1.
  \end{align*}
Then, we have
\begin{align*}
  \frac{1}{2\varphi^{-1}(1/2)}(h_K(u)+h_L(u))\leq h_{K+_\varphi L}(u).
\end{align*}
By the definition of Minkowski sum we have $h_K(u)+h_L(u)=h_{K+L}(u)$, so we obtain
$$\frac{1}{2\varphi^{-1}(1/2)}h_{K+L}(u)\leq h_{K+_\varphi L}(u).$$
Which means $\frac{1}{2\varphi^{-1}(1/2)}K+L\subseteq K+_\varphi L$.

(\romannumeral 3). Note that $h_{conv(K\cup L)}(u)=\max\{h_K(u), h_L(u) \}$.  Then

a), If  $h_{conv(K\cup L)}(u)=h_K(u)$, for some $u\in\Omega$, then
\begin{align*}
\varphi\left(\frac{h_K(u)}{h_{conv(K\cup L)}(u)}\right)+\varphi\left(\frac{h_L(u)}{h_{conv(K\cup L)}(u)}\right)=\varphi(1)+\varphi\left(\frac{h_L(u)}{h_K(u)}\right)\geq 1.
\end{align*}

b), If  $h_{conv(K\cup L)}(u)=h_L(u)$,  for some $u\in S^{n-1}\setminus\Omega$, then
\begin{align*}
\varphi\left(\frac{h_K(u)}{h_{conv(K\cup L)}(u)}\right)+\varphi\left(\frac{h_L(u)}{h_{conv(K\cup L)}(u)}\right)=\varphi(1)+\varphi\left(\frac{h_K(u)}{h_L(u)}\right)\geq 1.
\end{align*}
So we obtain $$\varphi\left(\frac{h_K(u)}{h_{conv(K\cup L)}(u)}\right)+\varphi\left(\frac{h_L(u)}{h_{conv(K\cup L)}(u)}\right)\geq 1,$$
 for $u\in S^{n-1}$. By the definition of Orlicz Minkowski sum we obtain
\begin{align*}
  h_{conv(K\cup L)}(u)\leq h_{K+_\varphi L}(u).
\end{align*}
So we have $conv(K\cup L)\subseteq K+_\varphi L$.

(\romannumeral 4). Since
\begin{align*}
  \varphi\left(\frac{h_K(u)}{h_{K+_\varphi L}(u)}\right)+ \varphi\left(\frac{h_L(u)}{h_{K+_\varphi L}(u)}\right)=1.
\end{align*}
The increasing of $\varphi$ implies
\begin{align*}
1\leq  \varphi\left(\frac{\max\{h_K(u),h_L(u)\}}{h_{K+_\varphi L}(u)}\right)+ \varphi\left(\frac{\max\{h_K(u),h_L(u)\}}{h_{K+_\varphi L}(u)}\right).
\end{align*}
Then, $$h_{K+_\varphi L}(u)\leq\frac{1}{\varphi^{-1}(1/2)}\max\{h_K(u),h_L(u)\}.$$
So we have $K+_\varphi L\subseteq \frac{1}{\varphi^{-1}(1/2)} conv(K\cup L).$
We complete the proof.
\end{proof}
Let $L\in\mathcal L^n_i$, and $K\in\mathcal K^n$, the orthogonal projection of  $K$ onto $L$ is denoted by $K|L$.

For the Orlicz Minkowski sum we have the following lemma.
\begin{lem}\label{orl-pro-rel-lem}
  Let $\varphi\in \mathcal C$, $K,K'\in \mathcal K^n_o$, and $L\in\mathcal L^n_i$, then
  \begin{align*}
    (K+_\varphi K')|L=K|L+_\varphi K'|L.
  \end{align*}
\end{lem}
\begin{proof}
  Let $x\in \mathbb R^n$, the following fact is obvious,
  \begin{align}\label{pro-equ}
   h_{K+_\varphi K'}(x|L)=h_{(K+_\varphi K')|L}(x).
  \end{align}
  By the definition of Orlicz Minkowski sum we have
  \begin{align*}
    \varphi\left(\frac{h_K(x|L)}{h_{K+_\varphi K'}(x|L)}\right)+ \varphi\left(\frac{h_{K'}(x|L)}{h_{K+_\varphi K'}(x|L)}\right)=1.
  \end{align*}
  By (\ref{pro-equ}) we have,
   \begin{align*}
    \varphi\left(\frac{h_{K|L}(x)}{h_{K+_\varphi K'}(x|L)}\right)+ \varphi\left(\frac{h_{K'|L}(x)}{h_{K+_\varphi K'}(x|L)}\right)=1.
  \end{align*}
 On the other hand,
   \begin{align*}
    \varphi\left(\frac{h_{K|L}(x)}{h_{(K|L+_\varphi K'|L)}(x)}\right)+ \varphi\left(\frac{h_{K'|L}(x)}{h_{(K|L+_\varphi K'|L)}(x)}\right)=1.
  \end{align*}
 Comparing with the above two formulas shows $$h_{K|L+_\varphi K'|L}(x)=h_{(K+_\varphi K')}(x|L)=h_{(K+_\varphi K')|L}(x).$$
  Which means $$(K+_\varphi K')|L=K|L+_\varphi K'|L.$$
  So we complete the proof.
  \end{proof}
  If $\varphi(t)=t^p$ $(p\geq 1)$, these results is obtained by Firey \cite{fir-p1962}.

 Let $L\in\mathcal L^n_i$,  write $B_{i,L}=B_n\cap L$ the unit ball contained in $L$.
  The following Lemma will be useful in the proof of the main results.
\begin{lem}\label{orl-add-lem}
  Let $\varphi\in \mathcal C$, $e_i,e_j$ be the orthogonal unit vectors in $\mathbb R^n$, $\tilde L=span\{e_i,e_j\}$,  then
  \begin{align}\label{orl-add-con-lem}
  [-e_i,e_i]+_\varphi[-e_j,e_j]\subseteq \frac{\sqrt 2}{2\varphi^{-1}(1/2)}B_{2,\tilde L}.\end{align}
\end{lem}
\begin{proof}
To prove $(\ref{orl-add-con-lem})$, we only need to show
  \begin{align*}
    h_{[-e_i,e_i]+_\varphi[-e_j,e_j]}(u)\leq \frac{\sqrt 2}{2\varphi^{-1}(1/2)},  \,\,\,\,\,\,\,\,\,\,\, for \,\,\,\,u\in S^{n-1}.
  \end{align*}
  For write simply, let $h_{[-e_i,e_i]+_\varphi[-e_j,e_j]}(u)=\lambda_u$. By the Orlicz Minkowski sum we have
  \begin{align*}
    \varphi\left(\frac{h_{[-e_i,e_i]}(u)}{\lambda_u}\right)+\varphi\left(\frac{h_{[-e_j,e_j]}(u)}{\lambda_u}\right)=1.
  \end{align*}
The symmetry of $[-e_i,e_i]+_\varphi[-e_j,e_j]$ ensure us it is enough to discuss it's support function with the parameter $\theta$ on interval $(0, \frac{\pi}{4}]$. Let $\omega(\theta)=e^{i\theta}$, then
  \begin{align*}
    \varphi\left(\frac{\cos\theta}{\lambda_{\omega(\theta)}}\right)+\varphi\left(\frac{\sin\theta}{\lambda_{\omega(\theta)}}\right)=1,
  \end{align*}
where $0< \theta \leq \frac{\pi}{4}$. The increasing of $\varphi$ and $\sin\theta\leq \cos\theta$ on interval $(0, \frac{\pi}{4}]$ implies
\begin{align*}
1= \varphi\left(\frac{\cos\theta}{\lambda_{\omega(\theta)}}\right)+\varphi\left(\frac{\sin\theta}{\lambda_{\omega(\theta)}}\right)\leq 2\varphi\left(\frac{\cos\theta}{\lambda_{\omega(\theta)}}\right).
\end{align*}
Then we have
\begin{align*}
 {\lambda_{\omega(\theta)}} \leq\frac{\cos\theta}{\varphi^{-1}(1/2)}\leq \frac{\sqrt 2}{2\varphi^{-1}(1/2)}.
\end{align*}
For other intervals, by the symmetry of $[-e_i,e_i]+_\varphi[-e_j,e_j]$, we have

\romannumeral  1:  If $\frac{\pi}{4}<\theta\leq\frac{\pi}{2}$, then $\lambda_{\omega(\theta)}=\lambda_{\omega(\frac{\pi}{2}-\theta)}$.

\romannumeral 2: If $\frac{\pi}{2}<\theta\leq\pi$, then $\lambda_{\omega(\theta)}=\lambda_{\omega(\pi-\theta)}$.

\romannumeral 3: If ${\pi}<\theta\leq2\pi$, then $\lambda_{\omega(\theta)}=\lambda_{\omega(2\pi-\theta)}$.\\
So we have $\lambda_{\omega_\theta}\leq\frac{\sqrt 2}{2\varphi^{-1}(1/2)}$ for $0<\theta\leq 2\pi$. We complete the proof.
\end{proof}

  If $\varphi(t)=t^p$, this result is obtained by Gordon and Junge \cite{gor-jun-vol1997}.

\section{Proof of main results}
In this section, we give the proofs of the main results.

{\it  Proof of Theorem \ref{out-rad-ine-thm}}: By Theorem \ref{lin-rel-thm} we have $\frac{1}{2\varphi^{-1}(1/2)} K+K'\subseteq K+_\varphi K'$, then
\begin{align*}
K+K'\subseteq2\varphi^{-1}(1/2)(K+_\varphi K').
\end{align*}
So we have
\begin{align*}
2\varphi^{-1}(1/2)R_1(K+_\varphi K')\geq R_1(K+K')\geq R_1(K)+R_1(K'),\\
2\varphi^{-1}(1/2)R_i(K+_\varphi K')\geq R_i(K+K')\geq\frac{1}{\sqrt 2}(R_i(K)+R_i(K')),
\end{align*}
$i=2,\cdots, n$, which shows the inequalities (\ref{out-rad-ine-1}) and (\ref{out-rad-ine-2}). Where we use the fact $R_1(K+K')\geq R_1(K)+R_1(K')$ and $R_i(K+K')\geq\frac{1}{\sqrt 2}(R_i(K)+R_i(K'))$, $i=2,\cdots, n$ (see \cite{gon-her-suc2012}).

To show the inequalities (\ref{out-rad-ine-1}) and (\ref{out-rad-ine-2}) are best possible, we find convex bodies satisfy the equality conditions.

For the equality of (\ref{out-rad-ine-1}), let $K=K'$, we have
\begin{align}\label{orl-dif-add}
  K+_\varphi K=\frac{1}{\varphi^{-1}(1/2)}K.
\end{align}
In fact, since $$\varphi\left(\frac{h_K(u)}{h_{K+_\varphi K}(u)}\right)+\varphi\left(\frac{h_K(u)}{h_{K+_\varphi K}(u)}\right)=1,$$
then,
 $$h_K(u)=\varphi^{-1}(1/2){h_{K+_\varphi K}(u)},$$
so (\ref{orl-dif-add}) holds.
Then we obtain,
\begin{align*}
  2\varphi^{-1}(1/2)R_1(K+_\varphi K) =2\varphi^{-1}(1/2)R_1\left(\frac{1}{\varphi^{-1}(1/2)}K\right)=R_1(K)+R_1(K).
\end{align*}
Which mean the equality in (\ref{out-rad-ine-1}) holds.

Next, for $i=2,\cdots,n-1$, we consider the convex bodies,
\begin{align*}
K=[-e_1,e_1]+\sum^n_{k=i+1}[-e_k,e_k],\,\,\,\,\,K'=[-e_2,e_2]+\sum^n_{k=i+1}[-e_k,e_k],
\end{align*}
i.e. the 0-symmetric $(n-i+1)$-cubes with edges parallel to the coordinate axes and with length 2 in the subspaces $L_j\subseteq\mathcal L^n_{n-i+1}$ (j=1,\,\,2). For $i=n$, we take $K=[-e_1,e_1]$, $K'=[-e_2,e_2]$.

Clearly, for $L\in \mathcal L^n_i$, $R(K|{L})\geq 1$, $R(K'|{L})\geq 1$.  Specially, if $L_0\in \mathcal L^n_i$ is generated by $\{e_1,\cdots,e_i\}$, then
\begin{align*}
  K|{L_0}=[-e_1,e_1],\,\,\,\,\,\,K'|{L_0}=[-e_2,e_2].
\end{align*}
Then we have
\begin{align*}
R_i(K)=\min_{L\in \mathcal L^n_i} R(K|{L})=R([-e_1,e_1])=1,\\
R_i(K')=\min_{L\in \mathcal L^n_i} R(K'|{L})=R([-e_2,e_2])=1.
\end{align*}
Since $dim(L_j\cap L)\geq 1$, then there exist $x\in K\cap L$ and $x'\in K'\cap L$ with $\lVert x\rVert_2, \lVert x'\rVert_2\geq 1$. By the symmetry of $K$  and $K'$, we may assume that $\langle x\cdot x'\rangle> 0$, then
\begin{align*}
\left \lVert\frac{x+x'}{2\varphi^{-1}(1/2)}\right \rVert_2\geq\frac{(\lVert x\rVert_2^2+\lVert x'\rVert_2^2)^{1/2}}{2\varphi^{-1}(1/2)}\geq\frac{2^{1/2}}{2\varphi^{-1}(1/2)}=\frac{\sqrt 2}{2\varphi^{-1}(1/2)}.
\end{align*}
Since $\frac{x+x'}{2\varphi^{-1}(1/2)}\in (K+_\varphi K')\cap L$. Using the fact that for $K\in \mathcal K^n_o$ and $L\in \mathcal L^n_i$, $$K\cap L\subseteq K|L,$$
we obtain,
\begin{align}\label{in-rad-pro-ine-1}
  R((K+_\varphi K')|L)\geq R((K+_\varphi K')\cap L)\geq \frac{\sqrt 2}{2\varphi^{-1}(1/2)}.
\end{align}
On the other hand, by Lemma \ref{orl-pro-rel-lem} we have
\begin{align*}
  (K+_\varphi K')|L_0=K|L_0+_\varphi K'|L_0=[-e_1,e_1]+_\varphi[-e_2,e_2].
\end{align*}
Lemma \ref{orl-add-lem} shows that
\begin{align*}
  (K+_\varphi K')|L_0=[-e_1,e_1]+_\varphi[-e_2,e_2]\subseteq \frac{\sqrt 2}{2\varphi^{-1}(1/2)}B_{2,\tilde{L}}.
\end{align*}
Where $\tilde L=span\{e_1,e_2\}$. So we have
\begin{align}\label{in-rad-pro-ine-2}
  R((K+_\varphi K')|L)\leq\frac{\sqrt 2}{2\varphi^{-1}(1/2)}.
\end{align}
Together with (\ref{in-rad-pro-ine-1}) and (\ref{in-rad-pro-ine-2}) we have
$$ R((K+_\varphi K')|L)=\frac{\sqrt 2}{2\varphi^{-1}(1/2)}.$$ Moreover,
\begin{align*}
  R_i(K+_\varphi K')=\min_{L\in \mathcal L^n_i}R((K+_\varphi K')|L)=\frac{\sqrt 2}{4\varphi^{-1}(1/2)}(R_i(K)+R_i(K')).
\end{align*}
Which gives the equality of (\ref{out-rad-ine-2}). So we complete the proof.{\hfill $\Box$}

Notice that, when $i=n$, namely, the circumradius $R(K+_\varphi K')$. By (\romannumeral 2) of Theorem \ref{lin-rel-thm} and the fact $R(K+K')\leq R(K)+R(K')$, the reverse inequality for the circumradius  holds
\begin{align*}
  R(K+_\varphi K')\leq R(K)+R(K').
\end{align*}
However, for $R_i(K+_\varphi K')$ ($i=1,2,\cdots,n-1$) there is no chance to get reverse inequalities.
\begin{prop}
Let $\varphi\in\mathcal C$ and $K, K'\in \mathcal K^n_o$, for $i=1,\cdots n-1$, there exists no constant $c>0$ such that
\begin{align*}
  cR_i(K+_\varphi K')\leq R_i(K)+R_i(K').
\end{align*}
\end{prop}
\begin{proof}
To show the non-existence of a reverse inequality, for $i=1,\cdots,n-1$. Take the convex bodies
  \begin{align*}
    K=[-e_{n-i+1},e_{n-i+1}] \,\,\,\,and \,\,\,\,K'=\sum^{n-i}_{k=1}[-e_k,e_k].
  \end{align*}
Let $\bar L_0$ and $\bar L'_0$ in $\mathcal L^n_i$ spanned by $\{e_{n-i},e_{n-i+2},\cdots, e_n\}$ and $\{e_{n-i+1},e_{n-i+2},\cdots, e_n\}$, respectively. Then
\begin{align*}
  K|{\bar L_{0}}=K'|{\bar L'_{0}}=\{0\}.
\end{align*}
So, $R_{i}(K)=R_{i}(K')=0$, i.e., $R_{i}(K)+R_{i}(K')=0$. On the other hand,
\begin{align*}
K+_\varphi K'\supseteq\frac{1}{2\varphi^{-1}(1/2)}(K+K')=\frac{1}{2\varphi^{-1}(1/2)}\sum^{n-i+1}_{k=1}[-e_k,e_k],
\end{align*}
that is $K+_\varphi K'$ contains an $(n-i+1)-$dimensional cube, which implies that $dim((K+_\varphi K')|L)\geq 1$ for $L\in\mathcal L^n_i$. Then, $R_i(K+_\varphi K')>0$, so there exists no constant $c>0$ such that
$$ cR_i(K+_\varphi K')\leq R_i(K)+R_i(K').$$
 We complete the proof.
\end{proof}
If take $\varphi (t)=t^p$, we obtain the following corollaries, which is obtained by Gonz\'alez and Hern\'andez Cifre \cite{gon-her-on2014}.
\begin{cor}
  Let  $K, K'\in \mathcal K^n_o$ and $p\geq 1$. Then
  \begin{align*}
&2^{\frac{p-1}{p}}R_1(K+_p K')\geq R_1(K)+R_1(K'),   \,\,\,\,\,\,\,\, for \,\,p\geq 1,\\
 &2^{\frac{3p-2}{2p}}R_i(K+_p K')\geq R_i(K)+R_i(K'), \,\,\,for \,\,1\leq p\leq 2,\,\,\, i=2,\cdots,n,\\
  &R_i(K+_p K')\geq \max\{R_i(K), R_i(K')\}, \,\,\,for \,\, p\geq 2,\,\,\, i=2,\cdots,n.
  \end{align*}
  All inequalities are best possible.
    \end{cor}
\begin{cor}
    Let  $K, K'\in \mathcal K^n_o$ and $p\geq 1$. Then
    \begin{align*}
      R_n(K+_pK')\leq R_n(K)+R_n(K')
    \end{align*}
    which is tight, and for any $i=1,\cdots,n-1$, there exists no constant $c>0$ such that $cR_i(K+_pK')\leq R_i(K)+R_i(K')$.
\end{cor}

 More specially, if $p=1$, it was shown in \cite{gon-her-suc2012}.

Now we deal with the inner radii $r_i$. The proof of Theorem \ref{in-rad-ine-thm} is similar with Theorem \ref{out-rad-ine-thm}.

{\it Proof of Theorem \ref{in-rad-ine-thm}:} By Theorem \ref{lin-rel-thm}, we have $\frac{1}{2\varphi^{-1}(1/2)} K+K'\subseteq K+_\varphi K'$, we obtain
\begin{align*}
2\varphi^{-1}(1/2)r_n(K+_\varphi K')\geq r_n(K+K')\geq r_n (K)+r_n(K'),\\
2\varphi^{-1}(1/2)r_i(K+_\varphi K')\geq r_i(K+K')\geq\frac{1}{\sqrt 2}(r_i(K)+r_i(K')),
\end{align*}
$i=1,2,\cdots, n-1$, which shows the inequalities (\ref{inn-rad-ine-1}) and (\ref{inn-rad-ine-2}).

Similarly, we will show  inequalities (\ref{inn-rad-ine-1}) and (\ref{inn-rad-ine-2}) are the best possible.
 Let $K=K'$, by (\ref{orl-dif-add}), we have
\begin{align*}
  2\varphi^{-1}(1/2)r_n(K+_\varphi K)=r_n(K)+r_n(K).
\end{align*}

For $i=1,2,\cdots,n-1$, let $j=2i-n$, if $2i\geq n$, and $j=0$ otherwise. We consider the $i-$dimensional linear subspaces
$\tilde L_0=span\{e_1,\cdots, e_{j},e_{j+1},\cdots, e_i\}$
 and $\tilde  L'_0=span\{e_1,\cdots, e_{j},e_{i+1},\cdots, e_{2i-j}\}$
 in $\mathcal L^n_i$. Let $B_{i,\tilde L_0}$ and  $B_{i,\tilde L'_0}$ be $i$-dimensional unit balls. It is clear that, $r_i(B_{i,\tilde L_0})=r_i(B_{i,\tilde L'_0})=1$.

 In the following, if we can show
 \begin{align}\label{orl-bal-add-ine}
   r((B_{i,\tilde L_0}+_\varphi B_{i,\tilde L'_0)})\cap L)\leq\frac{\sqrt 2}{2\varphi^{-1}(1/2)},
 \end{align}
 for $L\in \mathcal L^n_i$, then we obtain
 \begin{align*}
   r((B_{i,\tilde L_0}+_\varphi B_{i,\tilde L'_0})\cap  L)=\frac{\sqrt 2}{2\varphi^{-1}(1/2)}.
 \end{align*}
 In fact, since
 \begin{align*}
   \frac{\sqrt 2}{2\varphi^{-1}(1/2)}=\frac{\sqrt 2}{4\varphi^{-1}(1/2)}(r_i(B_{i,\tilde L_0})+r_i(B_{i,\tilde L'_0}))\leq r_i(B_{i,\tilde L_0}+_\varphi B_{i,\tilde L_0}).
 \end{align*}
 Note that $\tilde L_0+\tilde L'_0=\mathbb R^n$, then $dim(B_{i,\tilde L_0}+B_{i,\tilde L'_0})=n$, and $$dim((B_{i,\tilde L_0}+B_{i,\tilde L'_0})\cap  L)=i,$$
 for arbitrary $L\in\mathcal L^n_i$. Now if we can find $x\in bd((B_{i,\tilde L_0}+B_{i,\tilde L'_0})\cap \bar L)$ with $\lVert  x \rVert_2\leq\frac{\sqrt 2}{2\varphi^{-1}(1/2)}$ (where $bd(K)$ denotes the boundary of $K$), then we immediately get (\ref{orl-bal-add-ine}).

Now we find
such $x$, let $\tilde L''=span\{e_{j+1},\cdots, e_n\}$ if $j=2i-n\,\, (i.e. \,\,if\,\, 2i\geq n)$ then
\begin{align*}
  dim( L\cap\tilde  L'')&=dim L+dim\tilde L''-dim( L+\tilde L'')\\
  &\geq i-j=i-2i+n=n-i\geq 1.
\end{align*}
On the other hand, if $j=0$ then $\tilde L''=\mathbb R^n$, and so $ L\cap\tilde  L''=\ L$. Therefore, in both case $$dim((B_{i,\tilde L_0}+B_{i,\tilde L'_0})\cap L\cap \tilde L'')\geq 1,$$
which ensures the existence of a boundary point $z\in bd((B_{i,\tilde L_0}+_\varphi B_{i,\tilde L'_0})\cap L\cap \tilde L'')$. For any $z\in bd((B_{i,\tilde L_0}+_\varphi B_{i,\tilde L'_0})\cap \tilde L'')$ can  been expressed in the form $z=x+x'$, where $$x\in span\{e_{j+1},\cdots e_{i}\} \,\,and\,\, x'\in span\{e_{i+1},\cdots,e_{2i-j}\}.$$ Observe that $x,x'$ lie in orthogonal subspaces. Writing $u=\frac{z}{\lVert z\rVert_2}$, we have
\begin{align*}
\lVert z\rVert_2=\langle z\cdot u\rangle \leq h_{B_{i,\tilde L_0}+_\varphi B_{i,\tilde L'_0}}(u).
\end{align*}
Note that,
\begin{align*}
\nonumber h_{B_{i,\tilde L_0}}(u)&=\sup_{y\in B_{i,\tilde L_0}}\langle y\cdot u\rangle=\frac{1}{\lVert z\rVert_2}\sup_{y\in B_{i,\tilde L_0}}\langle y\cdot x\rangle=\frac{1}{\lVert z\rVert_2}\langle \frac{x}{\lVert x\rVert_2}\cdot x\rangle\\
&=h_{\left[-\frac{x}{\lVert x\rVert_2},\frac{x}{\lVert x\rVert_2}\right]}(u).
\end{align*}
Similarly, we have $h_{B_{i,\tilde L'_0}}(u)=h_{\left[-\frac{x'}{\lVert x'\rVert_2},\frac{x'}{\lVert x'\rVert_2}\right]}(u)$,
so we have
\begin{align*}
\lVert z\rVert_2\leq h_{\left[-\frac{x}{\lVert x\rVert_2},\frac{x}{\lVert x\rVert_2}\right]+_\varphi\left[-\frac{x'}{\lVert x'\rVert_2},\frac{x'}{|x'|_2}\right]}(u).
\end{align*}
By Lemma \ref{orl-add-lem}, we have
\begin{align*}
\left[-\frac{x}{\lVert x\rVert_2},\frac{x}{\lVert x\rVert_2}\right]+_\varphi\left[-\frac{x'}{\lVert x'\rVert_2},\frac{x'}{\lVert x'\rVert_2}\right]\subseteq \frac{\sqrt 2}{2\varphi^{-1}(1/2)}B_{2,\tilde L_2},
\end{align*}
where $\tilde  L_2=span\{\frac{x}{\lVert x\rVert},\frac{x'}{\lVert x'\rVert}\}$. So we have
\begin{align*}
 \lVert z\rVert_2\leq \frac{\sqrt 2}{2\varphi^{-1}(1/2)}.
\end{align*}
Then we have $$r_i(B_{i,\tilde L_0}+_\varphi B_{i,\tilde L'_0})=\frac{\sqrt 2}{2\varphi^{-1}(1/2)}.$$
So we complete the proof.

Similarly, by (\romannumeral 2) of Theorem \ref{lin-rel-thm} and the fact $r_1(K+K')\leq r_1(K)+r_1(K')$, the reverse inequality for the diameter  holds
\begin{align*}
  r_1(K+_\varphi K')\leq r_1(K)+r_1(K').
\end{align*}
For $r_i(K+_\varphi K')$ ($i=2,\cdots,n$) there is no chance to get reverse inequalities.
\begin{prop}
Let $\varphi\in\mathcal C$ and $K, K'\in \mathcal K^n_0$, for $i=2,\cdots,n$, there exists no constant $c>0$ such that
\begin{align*}
  cr_i(K+_\varphi K')\leq r_i(K)+r_i(K').
\end{align*}
\end{prop}
\begin{proof}
To show the non-existence of a reverse inequality, for $i=2,\cdots,n$, take the convex bodies
  \begin{align*}
    K=[-e_{1},e_{1}] \,\,\,\,and \,\,\,\,K'=\sum^{i}_{k=2}[-e_k,e_k].
  \end{align*}
clearly, $r_{i}(K)=r_{i}(K')=0$. On the other hand,
\begin{align*}
K+_\varphi K'\supseteq\frac{1}{2\varphi^{-1}(1/2)}(K+K')=\frac{1}{2\varphi^{-1}(1/2)}\sum^{i}_{k=1}[-e_k,e_k],
\end{align*}
So $r_i(K+_\varphi K')>0$, then there exists no constant $c>0$ such that $$ cr_i(K+_\varphi K')\leq r_i(K)+r_i(K').$$ We complete the proof.
\end{proof}
If we take $\varphi (t)=t^p$,  the following corollaries is obtained by Gonz\'alez and Hern\'andez Crfre \cite{gon-her-on2014}.
\begin{cor}
  Let  $K, K'\in \mathcal K^n_o$ and $p\geq 1$. Then
  \begin{align*}
&2^{\frac{p-1}{p}}r_n(K+_p K')\geq r_n(K)+r_n(K'),   \,\,\,\,\,\,\,\, for \,\,p\geq 1,\\
 &2^{\frac{3p-2}{2p}}r_i(K+_p K')\geq r_i(K)+r_i(K'), \,\,\,for \,\,1\leq p\leq 2, i=1,\cdots,n-1,\\
  &r_i(K+_p K')\geq \max\{r_i(K), r_i(K')\}, \,\,\,for \,\, p\geq 2, i=1,\cdots,n-1.
  \end{align*}
  All equalities are best possible.
    \end{cor}
\begin{cor}
    Let  $K, K'\in \mathcal K^n_o$ and $p\geq 1$. Then
    \begin{align*}
      r_1(K+_pK')\leq r_1(K)+r_1(K')
    \end{align*}
    which is tight, and for any $i=1,\cdots,n-1$, there exists no constant $c>0$ such that $cr_i(K+_pK')\leq r_i(K)+r_i(K')$.
\end{cor}

 More specially, if $p=1$, it was shown in \cite{gon-her-suc2012}.

\section {Radii of Orlicz difference body }
The Orlicz difference body of a convex body is defined as the Orlicz Minkowski addition of $K$ and it's reflection with origin, e.g.
\begin{align*}
  K+_\varphi(-K)
\end{align*}
which is a 0-symmetric body.
In \cite{che-xu-yan-rog2014} the Rogers Shephard inequality for the Orlicz difference body of a planar convex body is obtained.
In this section we interested in the behavior of the radii regarding Orlicz difference body.
\begin{thm}
  Let $\varphi\in \mathcal C$, and $K\in \mathcal K^n_o$, For all $i=1,2\cdots, n$, then
\begin{align}
\label{out-rad-ine-dif-1} \frac{\sqrt 2}{2\varphi^{-1}(1/2)}\sqrt{\frac{i+1}{i}}R_i(K)\leq R_i(K+_\varphi (-K)) \leq 2R_i(K),\\
\label{inn-rad-ine-dif-1}\frac{1}{\varphi^{-1}(1/2)}r_i(K)\leq r_i(K+_\varphi (-K))<2(i+1)r_i(K).
\end{align}
The inequality (\ref{out-rad-ine-dif-1}) are best possible.
\end{thm}
\begin{proof}
 By Theorem \ref{lin-rel-thm} we have
 \begin{align*}\frac{1}{2\varphi^{-1}(1/2)} K+(-K)\subseteq K+_\varphi (-K)\subseteq K+(-K).
 \end{align*}
Then
\begin{align*}
\frac{1}{2\varphi^{-1}(1/2)}R_i(K+(-K))\leq R_i(K+_\varphi(-K))\leq R_i (K+(-K)).
\end{align*}
By the fact $R_i(K+(-K))\leq 2R_i(K)$, and $\sqrt 2\sqrt{\frac{i+1}{i}}R_i(K)\leq R_i(K+(-K))$, we have
\begin{align}
 \frac{\sqrt 2}{2\varphi^{-1}(1/2)}\sqrt{\frac{i+1}{i}}R_i(K)\leq R_i(K+_\varphi(-K))\leq2R_i(K).
\end{align}
The similar with the proof of (\ref{inn-rad-ine-dif-1}) by using of the fact $$2r_i(K)\leq r_i(K+(-K))<2(i+1)r_i(K).$$

Now we show that the inequality (\ref{out-rad-ine-dif-1}) are the best possible.
Fix $i\in\{1,\cdots,n\}$, consider the convex body
\begin{align*}
  K=[0,e_1]+\sum^{n}_{k=i+1}[-e_k,e_k].
\end{align*}
Let $L\in\mathcal L^n_i$, and $\hat L=span\{e_1,\cdots,e_i\}\in\mathcal L^n_i$, it is clearly that
\begin{align*}
  R_i(K)=\min\{R(K|L):L\in \mathcal L^n_i\}=R(K|\hat  L)=R([0,e_1])=\frac{1}{2},
\end{align*}
here, if $i=n$ we take $K=[0,e_1]$.
On the other hand, notice that
\begin{align*}
 R_i[(K&+_\varphi(-K))]=R[(K|L)+_\varphi (-K)|L]\\
 &= R([0,e_1]+_\varphi[-e_1,0])\\
 &=R([-e_1,e_1])=1.
\end{align*}
In fact, by Theorem \ref{lin-rel-thm}, we have
\begin{align*}
 [0,e_1]+_\varphi[-e_1,0]\subseteq [0,e_1]+[-e_1,0]=[-e_1,e_1].
\end{align*}
On the other hand we have $$[-e_1,e_1]=conv([0,e_1]\cup[-e_1,0])\subseteq [0,e_1]+_\varphi [-e_1,0].$$
So we have
\begin{align*}
  R_i[K+_\varphi(-K)]=R([-e_1,e_1])=1=2R_i(K).
\end{align*}

If $i=n$, let $K_n$ be the $n-$dimensional simplex embedded in $\mathbb R^{n+1}$, lying in the hyperplane $\{x=(x_1,\cdots,x_{n+1})\in \mathbb R^{n+1}: \sum^{n+1}_{j=1}x_j=0\}$, given by
\begin{align*}
  K_n=conv\{p_k: p_{kk}=\frac{n}{n+1}, p_{kj}=\frac{-1}{n+1} \,\,for\,\,j\neq k, \,\,k=1,2,\cdots n+1\}
\end{align*}
Note that $R_n(K_n)=\sqrt{\frac{n}{n+1}}$.

Since $K_n+_\varphi (-K_n)$ is 0-symmetric convex body, then
$$R_n(K_n+_\varphi(-K_n))=\max\{\ h_{K_n+_\varphi(-K_n)}(u): \lVert u\rVert_2=1 \,\,and \,\, \sum_{j=1}^{n+1}u_j=0\},$$
Notice that the value of the support function of a convex body at any vector is attained in an extreme point (see \cite{gru-con2007}). So we consider the vertices of $K_n$. Since
\begin{align*}
  \langle p_k\cdot u\rangle=\frac{n}{n+1}u_k-\frac{1}{n+1}\sum_{j\neq k}u_j=u_k.
\end{align*}
Then $h_{K_n}(u)=\max\{u_1,\cdots,u_{n+1}\}$.
Without loss of generality we may assume that $u_1\geq u_2\cdots\geq u_{n+1}$. Hence by the Orlicz Minkowski addition we have
\begin{align*}
\varphi\left(\frac{h_{K_n}(u)}{h_{K_n+_\varphi(-K_n)}(u)}\right)+\varphi\left(\frac{h_{K_n}(-u)}{h_{K_n+_\varphi(-K_n)}(u)}\right)=1,
\end{align*}
Notice that $\varphi$ is increasing and observe that $u_1\geq 0$ and $u_{n-1}\leq 0$. Then the maximum of $h_{K_n+_\varphi(-K_n)}(u)$ under the conditions $\rVert u\lVert_2=1$ and $\sum^{n+1}_{j=1}u_j=0$, is attained in the point $(1/\sqrt 2, 0,\cdots,-1/\sqrt 2)$, therefore,
\begin{align*}
  R_n(K_n+_\varphi(-K_n))=\max\{h_{K_n+_\varphi(-K_n)}(u):u\in S^{n-1}\}=\frac{\sqrt 2}{2\varphi^{-1}(1/2)}.
\end{align*}
Then we have
\begin{align*}
R_n(K_n+_\varphi(-K_n))=\frac{\sqrt 2}{2\varphi^{-1}(1/2)}=\frac{\sqrt 2}{2\varphi^{-1}(1/2)}\sqrt{\frac{n+1}{n}}R_n(K_n).
\end{align*}
So we get the equality condition for $i=n$.

If $i<n$, we take the $i-$dimensional simplex $K_i$ and consider the convex body $K=K_i+cM_{n-i}$, where $M_{n-i}$ is a $(n-i)-$dimensional unit cube in (aff $K_i$)$^\bot$, and $c>0$ is sufficiently large such that $R_i(K+_\varphi(-K))=R(K_i+_\varphi K_{i})$ and $R_i(K)=R(K_i)$. By the same argument we obtain the case of $i<n$. So we complete the proof.

\end{proof}

\bibliographystyle{amsplain}

\end{document}